\documentclass{amsart}
\usepackage{color}

\newtheorem{lemma}{Lemma}
\newtheorem{corollary}{Corollary}
\newtheorem{proposition}{Proposition}
\theoremstyle{conjecture}

\theoremstyle{definition}

\theoremstyle{question}

\theoremstyle{questions}

\theoremstyle{remark}

\theoremstyle{remarks}

\theoremstyle{example}
\newtheorem*{example}{Examples}
\numberwithin{equation}{section}

\begin{document}
\title{Sobolev Maps into Compact Lie Groups and Curvature}

\author{Andres Larrain-Hubach}

\author{Doug Pickrell}
\email{pickrell@math.arizona.edu}

\begin{abstract} These are notes on seminal work of Freed, and subsequent developments, on the curvature properties
of (Sobolev Lie) groups of maps from a Riemannian manifold into a compact Lie group. We are mainly interested in critical cases
which are relevant to quantum field theory. For example Freed showed that, in a necessarily qualified sense, the quotient space $W^{1/2}(S^1,K)/K$ is a (positive constant) Einstein `manifold' with respect to the
essentially unique $PSU(1,1)$ invariant metric, where $W^{s}$ denotes maps of $L^2$ Sobolev order $s$. In a similarly qualified sense, and in addition making use of the Dixmier trace/Wodzicki residue, we show that for a Riemann surface $\Sigma$, $W^1(\Sigma,K)/K$ is a (positive constant) Einstein `manifold' with respect to the essentially unique conformally invariant metric. As in the one dimensional case, invariance implies Einstein, but the sign of the Ricci curvature has to be computed. Because of the qualifications involved in these statements, in practice it is necessary to consider curvature for $W^s(\Sigma,K)$ for $s$ above the critical exponent, and limits. The formula we obtain is surprisingly simple.

\end{abstract}

\maketitle

\section{Introduction}

In these notes (following many other authors) we consider curvature for left invariant Riemannian metrics on mapping groups (and some related homogeneous spaces). Throughout we suppose that $\Sigma$ is a compact Riemannian manifold and $K$ is a compact Lie group. The examples of basic interest are the following:

\bigskip

Example I. If $s>dim(\Sigma)/2$, then $\mathbf K=W^s(\Sigma;K)$, the space of continuous maps which are smooth of order $s$ in the $L^2$ Sobolev sense, is a Hilbert Lie group (see \cite{EM}). Fix a Riemannian metric on the Lie algebra, $W^s(\Sigma,\mathfrak k)$,
$$ \langle x,y\rangle:= \int_{\Sigma} \langle\langle P^{s}x,y\rangle\rangle dV$$
where $\langle\langle \cdot,\cdot\rangle\rangle$ denotes an $Ad$ invariant inner product on $\mathfrak k$, $P$ is a positive Laplace
type operator (e.g. $P=\Delta+m_0^2$, where $\Delta$ is the Laplace operator), and $dV$ is the volume element, corresponding to the Riemannian metric on $\Sigma$. This defines a complete left invariant Riemannian metric on $\mathbf K$ (If $P=\Delta+m_0^2$ and $m_0=0$, then we think of this as a complete left invariant Riemannian metric on the quotient $\mathbf K/K$).

\bigskip

Example II. As in Example I, suppose $s>dim(\Sigma)/2$. Given a finite configuration of points $V\subset \Sigma$, there is an evaluation map
$$W^s(\Sigma;K)\to \prod_V K_v:g \to (g_v)$$
and we can consider the left invariant Riemannian structure on the image for which the map is a Riemannian submersion.
This is a statistical mechanical analogue of Example I. The inner product on the Lie algebra $\prod_V\mathfrak k_v$ is given by
$$\langle x,y\rangle=\sum_V  G^{-1}_V(v,w)\langle\langle x(v),y(w)\rangle\rangle$$
where in this context $G^{-1}_V$ denotes the inverse to the matrix $(G(v,w))_{v,w\in V}$ and $G$ is the Green's function for $P^s$ restricted to $V$, i.e.
$$(P^{-s}f)(v)=\int_{\Sigma}G(v,w)f(w)dV(w), \qquad f\in W^{-s}(\Sigma)$$

\bigskip

Example III. When $s=dim(\Sigma)/2$, the critical exponent, $W^s(\Sigma;K)$ (which now $g\in W^s(\Sigma,K)$ must be interpreted as an equivalence class of maps - a generic class no longer has a continuous representative) is a topological group and (at least in the case $\Sigma=S^1$, conjecturally in general) a topological Hilbert manifold, but it is not in general a Lie group (when $\mathfrak k$ is nonabelian, $W^s(\Sigma,\mathfrak k)$ is not a Lie algebra: the commutator fails to be closed, because for critical $s$, a generic $W^s$ map $x:\Sigma \to\mathfrak k$ is not bounded). Nonetheless one can heuristically calculate curvature for the quotient $W^{dim(\Sigma)/2}(\Sigma;K)/K$, when it is equipped with the natural translation and conformally invariant metric (For a surface the inner product at the basepoint is given as in Example I, with $P=\Delta$; in general $P^{dim(\Sigma)/2}$ is replaced by a conformally invariant (Paneitz or critical GJMS) operator; see \cite{Morpurgo}).

\bigskip

As we will explain next, from a naive point of view, we are primarily interested in Example III (and especially for $\Sigma$ a circle or a surface). However, in reality we need to consider the Examples I and II and limits.

\subsection{Motivation (from Quantum Field Theory)}\

1. For maps from one Riemannian manifold to another, the natural (local) energy functional is the Dirichlet functional; in our context
$$\mathcal E:W^1(\Sigma,K)\to\mathbb R:g \to \frac12\int_{\Sigma}\langle\langle g^{-1}dg\wedge *g^{-1}dg\rangle \rangle$$
For this energy functional the critical dimension is $dim(\Sigma)=2$, in which case $\mathcal E$ is conformally invariant (i.e. depends only on the conformal structure induced by the metric on the domain).
For a surface $\Sigma$ and a simply connected target $K$, a fundamental unsolved problem of science is to correctly define
the associated quantum theory, which naively amounts to making sense of the (un-normalized) Feynman measure
$$exp(-\frac1{t}\mathcal E(g))\prod_{v\in\Sigma}|dg_v|$$
This sigma model is reputed to be asymptotically free - and hence is expected to exist - but the mathematical elucidation of this
is incomplete. A potentially reasonable way to regularize the normalized version of this Feynman expression
is to consider the (mathematically well-defined) heat kernel measures for the groups $W^s(\Sigma,K)$ for $s>1$ (see e.g. \cite{Pi3}). It is well-known that the behavior of a heat kernel on a Riemannian manifold is intimately related to the Ricci curvature of the metric. For loop groups this is effectively exploited by Driver and Gordina in \cite{DG}, to prove that loop group heat kernel measures are translation quasi-invariant; the key is to show that there is a uniform lower bound for Ricci curvature for the groups in II above, associated to a circle. In the critical conformally invariant case of a surface, the expectation is that the heat kernel measures for $W^s(\Sigma,K)$, $s>1$, marginally
fail to be translation quasi-invariant. Renormalization group arguments (which predict asymptotic freedom) suggest that ultimately one has to understand a limit as $s\downarrow 1$ (for the un-normalized measures) when the coupling parameter $t=ln(s)$ (see e.g. lecture 3 of \cite{Gawedzki}).\

2. For the Yang-Mills functional the natural domain consists of $W^1$ connections:
$$YM:W^1\Omega^1(\Sigma;\mathfrak k)\to \mathbb R:A\to \frac12\int_{\Sigma}\langle\langle F_A\wedge * F_A\rangle\rangle$$
For this energy functional the critical dimension is $dim(\Sigma)=4$, in which case $\mathcal YM$ is conformally invariant.
Unsurprisingly, the relevant gauge group is $W^s(\Sigma,K)$, where $s=dim(\Sigma)/2=2$ is the critical exponent.
It is again potentially important to understand curvature (To our knowledge
the original observation along these lines is due to I.M. Singer. He observed that because curvature increases along submersions, the space of gauge equivalence classes is non-negatively curved, and he suggested this might shed light on existence of a mass gap).\

For our purposes, the upshot of this physics digression is that, to focus on spheres for simplicity, there is a great deal of interest in understanding the geometry of the critical sequence (of, in a qualified sense, Riemannian submersions)
$$... \to W^{d/2}(S^d,K)/K \to ... \to W^{1}(S^2,K)/K \to W^{1/2}(S^1,K)/K $$
(where these spaces are equipped with the natural conformally invariant metrics), and the zeta type approximations to these spaces for $s>dim/2$, the critical exponent (for more general spaces in place of spheres, less symmetric metrics arise).

These notes are largely motivated by the following observation: in the special case that $\Sigma$ is a sphere with the standard conformal structure, by general symmetry considerations the space $W^{dim(\Sigma)/2}(\Sigma;K)/K$ has to be (in a heuristic sense) an Einstein manifold, provided that one can make sense of the Ricci curvature in some reasonable symmetry preserving way. As observed by Freed, Ricci curvature does not make sense in this context in the naive sense, but he found a very reasonable way to regularize Ricci curvature in the one dimensional case. In a qualified sense, he observed that $W^{1/2}(S^1,K)/K$ is (positive constant) Kahler-Einstein.

\subsection{What We Aim For}\

1. For the groups $\mathbf K=W^s(\Sigma;K)$ and metrics as in Example I, we first observe that the curvature operator, $z\to R(x,y)z$, is a pseudo-differential (psd) operator of order $\le -2$ (this is a technical improvement on previous results).

2. Ricci curvature $Ric(y,z)$ is the trace of the operator $x\to R(x,y)z$, provided that one can make sense of the trace. Freed observed that the order of this operator
is $-1$, and consequently the trace is never defined in the naive sense for the groups in Example I. However Freed also observed that if one first takes the trace over $\mathfrak k$, then one obtains a scalar psd operator of order $\le -2$, and consequently, if $\Sigma$ is one dimensional (and not otherwise), one can take a trace of this scalar psd operator to obtain a reasonable regularization for Ricci curvature. This procedure has been refined and fully justified by many subsequent developments (see especially \cite{Driver} and \cite{Inahama}). In the critical two dimensional case, it seems at least a priori reasonable to replace the ordinary trace of the scalar psd operator by the Wodzicki residue/Dixmier trace, to define Ricci curvature in this case.

3. Suppose that $\Sigma=S^1$ and $K$ is simply connected.  For $s=1/2$ the quotient $W^{1/2}(S^1,K)/K$ is Kahler-Einstein with
positive Ricci curvature. For $s=1$ the Ricci curvature for $W^1(S^1,K)/K$ is negative definite, but not Einstein. For $s\ne 1$ we have observed numerically that for sufficiently small mass, the Ricci curvature for $W^s(S^1,K)$ is positive definite, but not bounded below by a positive constant. We have not managed to prove this rigorously, and we hope to fill this gap at a later time.

4. Suppose that $\Sigma$ is a Riemannian surface. Using the Wodzicki residue/Dixmier trace to define Ricci curvature, we find that
the Ricci curvature for $W^s(\Sigma,K)$ equals
$$Ric(y,z)=\pi s^2\int_{\Sigma}-\kappa( dy\wedge *dz)$$
where $\kappa$ denotes the Killing form (so that ($-\kappa$ is positive on $\mathfrak k$). In particular, according to our interpretation
of Ricci curvature, the curvature is positive for all $s$, and for $s=1$, $W^1(\Sigma,K)/K$ is Einstein with positive constant, with respect to the conformally invariant metric.

5. It is an important question to determine whether there is a uniform lower bound for Ricci curvature for configurations of compact groups on a Riemannian surface, for fixed $s>1$. This seems unlikely, but again we have failed to find a proof of this.

\subsection{Plan of the Paper}\

In Section \ref{curvop} we consider the curvature operator, $z\to R(x,y)z$. We first consider a general Lie group with a left invariant metric. We secondly specialize to the case when the metric can be written in terms of a biinvariant background. In this context our calculations apply equally to Sobolev maps into compact groups and to configurations of compact groups. We then specialize to mapping groups.

In Section \ref{riccurv} we consider the operator $x\to R(x,y)z$. We are ultimately interested in taking the trace of this operator, in some sense, which is some version of Ricci curvature. We again begin by considering a general Lie group with a left invariant metric.
We then specialize to the case when the metric can be written in terms of a biinvariant background. In this context our calculations
apply equally to Sobolev maps into compact groups and to configurations of compact groups. We then specialize to mapping groups. When the domain is a surface, we define Ricci curvature using the Dixmier trace/Wodzicki trace. This is a severe regularization. In particular we will see that the resulting regularized Ricci curvature does not depend on the mass term $m_0$.

\subsection{Notation}

Throughout these notes $\Sigma$ denotes a compact Riemannian manifold (we will often specialize to the cases of a circle or surface), $\Delta=d^*d$ is the nonnegative Laplacian for functions on $\Sigma$, where $d$ is the exterior derivative, and so on.
$K$ denotes a compact Lie group with Lie algebra $\mathfrak k$, $\langle\langle\cdot,\cdot\rangle\rangle$ denotes a fixed $Ad$ invariant inner product on $\mathfrak k$, and $\kappa$ denotes the Killing form on $\mathfrak k$. $W^s(\Sigma)$ denotes the topological vector space of equivalence classes of real functions which are smooth order $s$ in the $L^2$ Sobolev sense. For the topology of nonlinear spaces of Sobolev maps, $W^s(\Sigma;X)$, see \cite{Brezis} (and a book in preparation).

\section{Left Invariant Metrics and the Curvature Operator}\label{curvop}

Let $\mathbf K$ denote a Lie group. Fix an inner product $\langle\cdot,\cdot\rangle$ on the Lie algebra $Lie(\mathbf K)$,
which we identify with the corresponding left invariant metric on $\mathbf K$. For left invariant vector fields $x,y$ on $\mathbf K$,
the Levi-Civita connection is given by
$$\nabla_x(y)=\frac12([x,y]-ad_x^*(y)-ad_y^*(x))$$
where $ad_x^*$ denotes the adjoint of $ad_x$ with respect to the inner product,
i.e.
$$\langle\nabla_x(y),z\rangle=\frac12(\langle[x,y],z\rangle-\langle [x,z],y\rangle-\langle [y,z],x\rangle)$$

Our convention for the Riemann curvature is
$$R(x,y)=[\nabla_x,\nabla_y]-\nabla_{[x,y]}$$

More explicitly $R(x,y)z$ equals
\begin{equation}\label{curvature1}\frac 14([x,[y,z]]-ad_x\circ ad_y^*(z)-ad_x\circ ad_z^*(y)-ad_x^*([y,z])+ad_x^*\circ ad_y^*(z)+ad_x^*\circ ad_z^*(y)\end{equation}
$$-ad_{([y,z]-ad_y^*(z)-ad_z^*(y))}^*(x))$$
$$-\frac 14([y,[x,z]]-ad_y\circ ad_x^*(z)-ad_y\circ ad_z^*(x)-ad_y^*([x,z])+ad_y^*\circ ad_x^*(z)+ad_y^*\circ ad_z^*(x))$$
$$-ad_{([x,z]-ad_x^*(z)-ad_z^*(x))}^*(y))-\frac12([[x,y],z]-ad_{[x,y]}^*(z)-ad_z^*([x,y]))$$\\
or equivalently, by organizing parts of this in terms of commutators,
\begin{equation}\label{curvoperator2}\frac 14(-ad_{[x,y]}(z)-ad_{[y,z]}^*(x)+ad_{[x,z]}^*(y))\end{equation}
$$+\frac 14 (-[ad_x,ad_y^*](z)+[ad_y,ad_x^*](z)+[ad_x^*,ad_y^*](z))$$
$$+\frac 14(-ad_x\circ ad_z^*(y)+ad_x^*\circ ad_z^*(y))-\frac 14(-ad_y\circ ad_z^*(x)+ad_y^*\circ ad_z^*(x))$$
$$+\frac 14(ad_{ad_y^*(z)}^*(x)+ad_{ad_z^*(y)}^*(x))-\frac 14(ad_{ad_x^*(z)}^*(y)+ad_{ad_z^*(x)}^*(y))$$
$$-\frac12(-ad_{[x,y]}^*(z)-ad_z^*([x,y]))$$

\begin{example} As a check on conventions: Suppose that $\mathbf K$ is a compact Lie group with Ad-invariant inner product. Then
$\nabla_x=\frac12 ad_x$, $R(x,y)=-\frac14 ad_{[x,y]}$, (unnormalized) sectional curvature $\langle R(x,y)y,x\rangle=\frac14|[x,y]|^2$
$$Ric(y,z)=\mathrm{trace}(x\to R(x,y)z)=\sum \langle R(x_i,y)z,x_i\rangle$$
where $\{x_i\}$ is an orthonormal basis; this equals $-2\dot g\kappa(y,z)$ when the Lie algebra is simple and $\dot g$ denotes the dual Coxeter number.

\end{example}

\subsection{The Curvature Operator}

In this subsection we assume that the inner product is of the form
$$\langle x,y\rangle=\langle\langle G^{-1} x,y\rangle\rangle$$
where $\langle\langle \cdot,\cdot\rangle\rangle$ is an $Ad$ invariant inner product (To explain the notation, in later sections $G$ will be a Green's function, and $G^{-1}$ will be a positive pseudodifferential operator). In this case the adjoint of $ad_x$ is given by
$$ad_x^*=-G\circ ad_x \circ G^{-1}$$
This is equivalent to
$$ad_x^*(y)=G\circ ad_{G^{-1}y}(x)$$

With this assumption $R(x\wedge y)z$ equals
$$\frac 14 (-ad_{[x,y]}(z)-[ad_x,ad_y^*](z)+[ad_y,ad_x^*](z)+[ad_x^*,ad_y^*](z))$$
$$+\frac 14(-G\circ ad_{G^{-1}x}\circ ad_y(z)+G\circ ad_{G^{-1}y}\circ ad_x(z))$$
$$+\frac 14(-ad_x\circ G\circ ad_{G^{-1}y}(z)-G\circ ad_x\circ ad_{G^{-1}y}(z))$$
$$-\frac 14(-ad_y\circ G\circ ad_{G^{-1}x}(z)-G\circ ad_y\circ ad_{G^{-1}x}(z))$$
$$+\frac 14(-G\circ ad_{G^{-1}x}\circ G\circ ad_y\circ G^{-1}(z) +G \circ ad_{G^{-1}x}\circ G \circ ad_{G^{-1}y}(z))$$
$$-\frac 14(-G\circ ad_{G^{-1}y}\circ G\circ ad_x\circ G^{-1}(z) +G \circ ad_{G^{-1}y}\circ G \circ ad_{G^{-1}x}(z))$$
$$-\frac12(-ad_{[x,y]}^*(z)-G\circ ad_{G^{-1}[x,y]}(z))$$

Now we gather terms that involve $ad_x$ and $ad_y$ and their adjoints:
$$=\frac 14 (-ad_{[x,y]}(z)-[ad_x,ad_y^*](z)-[ad_x^*,ad_y](z)+[ad_x^*,ad_y^*](z)+2ad_{[x,y]}^*(z))$$
$$+\frac 14(-G\circ ad_{G^{-1}x}\circ ad_y(z)+ad_y\circ G\circ ad_{G^{-1}x}(z)-G\circ ad_{G^{-1}x}\circ G\circ ad_y\circ G^{-1}(z)$$
$$+G\circ ad_y\circ ad_{G^{-1}x}(z))$$
$$+\frac 14(+G\circ ad_{G^{-1}y}\circ ad_x(z)-ad_x\circ G\circ ad_{G^{-1}y}(z)+G\circ ad_{G^{-1}y}\circ G\circ ad_x\circ G^{-1}(z)$$
$$-G\circ ad_x\circ ad_{G^{-1}y}(z))$$
$$+\frac 14( +G \circ ad_{G^{-1}x}\circ G \circ ad_{G^{-1}y}(z))-G \circ ad_{G^{-1}y}\circ G \circ ad_{G^{-1}x}(z))$$
$$+\frac12 G\circ ad_{G^{-1}[x,y]}(z)$$

Now we condense the expression slightly using commutators:
$$=\frac 14 (-ad_{[x,y]}(z)+[ad_x,G\circ ad_y\circ G^{-1}](z)+[G\circ ad_x\circ G^{-1},ad_y](z)$$
$$+[G\circ ad_x\circ G^{-1},G\circ ad_y\circ G^{-1}](z)-2G\circ ad_{[x,y]}\circ G^{-1}(z))$$
$$+\frac 14(-G\circ ad_{[G^{-1}x,y]}+ad_y\circ G\circ ad_{G^{-1}x}(z)-G\circ ad_{G^{-1}x}\circ G\circ ad_y\circ G^{-1}(z)$$
$$+\frac 14(+G\circ ad_{[G^{-1}y,x]}-ad_x\circ G\circ ad_{G^{-1}y}(z)+G\circ ad_{G^{-1}y}\circ G\circ ad_x\circ G^{-1}(z)$$
$$+\frac 14 [G \circ ad_{G^{-1}x}, G \circ ad_{G^{-1}y}](z)$$
$$+\frac12 G\circ ad_{G^{-1}[x,y]}(z)$$

\begin{equation}\label{curvop1}=\frac 14 (-ad_{[x,y]}(z)+[ad_x,G\circ ad_y\circ G^{-1}](z)+[G\circ ad_x\circ G^{-1},ad_y](z)-G\circ ad_{[x,y]}\circ G^{-1}(z))\end{equation}
$$+\frac 14(-G\circ ad_{[G^{-1}x,y]}+ad_y\circ G\circ ad_{G^{-1}x}(z)-G\circ ad_{G^{-1}x}\circ G\circ ad_y\circ G^{-1}(z)$$
$$+\frac 14(+G\circ ad_{[G^{-1}y,x]}-ad_x\circ G\circ ad_{G^{-1}y}(z)+G\circ ad_{G^{-1}y}\circ G\circ ad_x\circ G^{-1}(z)$$
$$+\frac 14 [G \circ ad_{G^{-1}x}, G \circ ad_{G^{-1}y}](z)$$
$$+\frac12 G\circ ad_{G^{-1}[x,y]}(z)$$

It is possible to express the 2nd and 3rd lines using commutators, but there does not seem to be any advantages to do this.

\subsection{The Example $\mathbf K=W^s(\Sigma;K)$}\

\begin{proposition}\label{prop1}  Suppose that $s\ge 1$. Then the psd order of $z\to R(x,y)z$ is $\le -2$.

\end{proposition}

This generalizes the result of the first author in \cite{LH}.

\begin{proof} In the expression (\ref{curvop1}) for the curvature operator, it is clear that all of the terms have order $\le -2s$,
with the exception of those appearing in the first line. Thus we focus on the terms in the first line.

Suppose initially that $x=X\otimes a$ and $y=Y\otimes b$, where $X$ is a function,$a\in \mathfrak k$,... .
To simplify the notation, identify $X$ with the corresponding multiplication operator, $a$ with $ad_a$, and so on. Then
the first line of (\ref{curvop1})
$$-ad_{[x,y]}+[ad_x,G\circ ad_y\circ G^{-1}]+[G\circ ad_x\circ G^{-1},ad_y]-G\circ ad_{[x,y]}\circ G^{-1}$$
$$=-X\circ Y\circ [a,b]+X\circ G\circ Y\circ G^{-1} \circ ab-G\circ Y\circ G^{-1} \circ X\circ ba+G\circ X\circ G^{-1} \circ Y ab-Y\circ
G\circ X\circ G^{-1} ba-G\circ X\circ Y\circ G^{-1} [a,b]$$
$$=(-X\circ Y+X\circ G\circ Y\circ G^{-1}+G\circ X\circ G^{-1}\circ Y-G\circ X\circ Y\circ G^{-1})ab$$
$$-(-X\circ Y+Y\circ G\circ X\circ G^{-1}+G\circ Y\circ G^{-1}\circ X-G\circ X\circ Y\circ G^{-1})ba$$
$$=(-X\circ Y+X\circ G\circ (G^{-1}\circ Y+[Y,G^{-1}])+G\circ X\circ G^{-1}\circ Y-G\circ X\circ (G^{-1}\circ Y+[Y,G^{-1}]))ab+(...)ba$$
$$=(X\circ G\circ [Y,G^{-1}])-G\circ X\circ [Y,G^{-1}])ab+(...)ba$$
$$=[X,G]\circ [Y,G^{-1}] ab-[Y,G]\circ [X,G^{-1}]ba$$
$$=[ad_x,G]\circ [ad_y,G^{-1}]-[ad_y,G]\circ [ad_x,G^{-1}]$$
$$=G\circ [G^{-1},ad_x]\circ G\circ [ad_y,G^{-1}]-G\circ [G^{-1},ad_y]\circ G\circ [ad_x,G^{-1}]$$
\begin{equation}\label{line1}=[G\circ [G^{-1},ad_y], G\circ [G^{-1},ad_x]]\end{equation}
This can also be written as
\begin{equation}\label{line2}[G,ad_x][G^{-1},ad_y]-[G,ad_y][G^{-1}ad_x]\end{equation}
Both the expressions (\ref{line1}) and (\ref{line2}) show clearly that the first line of (\ref{curvop1}) has order $\le -2$.

A general $x$ can be written as a linear combination of the factored elements $X\otimes a$. Since the operator we are considering
depends on $x$ and $y$ in a bilinear way, this completes the proof.
\end{proof}

Using the calculation in the proceeding proof, we can somewhat condense the expression for the curvature operator.
These calculations are valid also for the statistical mechanical case $\prod_V K_v$, because we can replace
functions on $\Sigma$ by functions on $V$.

\begin{proposition} For either $W^s(\Sigma,K)$, $s>dim(\Sigma)/2$, or $\prod_V K_v$, $R(x,y)$ equals
$$\frac14([G\circ ([G^{-1},ad_x]-ad_{G^{-1}(x)}),G\circ ([G^{-1},ad_y]-ad_{G^{-1}(y)}]-2[G\circ ad_x\circ G^{-1},G\circ ad_{G^{-1}(y)}]$$
$$
+2[G\circ ad_y\circ G^{-1},G\circ ad_{G^{-1}(x)}]+2G\circ ad_{G^{-1}([x,y])})$$
\end{proposition}

\section{Ricci Curvature I: In General}\label{riccurv}

We now consider the question of how to calculate Ricci curvature for $\mathbf K$, which assuming it makes sense, is given by
$$Ric(y,z)=trace(x\to R(x,y)(z))$$ We postpone the question of how to define the trace - we will have to do this on a case by case basis.

Suppose initially that we do not make any assumptions about the form of the inner product on the Lie algebra of $\mathbf K$.
Using the expression (\ref{curvature1}) for curvature,
as an operator on $x$, $R(x,y)z$ equals
$$\frac 14(ad_{ad_y^*(z)}(x)+ad_{ad_z^*(y)}(x)-ad_x^*([y,z])+ad_x^*\circ ad_y^*(z)+ad_x^*\circ ad_z^*(y)$$
$$-ad_{([y,z]-ad_y^*(z)-ad_z^*(y))}^*(x))$$
$$-\frac 14(-ad_y\circ ad_x^*(z)-ad_y\circ ad_z^*(x)-ad_y^*([x,z])+ad_y^*\circ ad_x^*(z)+ad_y^*\circ ad_z^*(x))$$
$$-ad_{([x,z]-ad_x^*(z)-ad_z^*(x))}^*(y))$$
$$-\frac 14 ad_z\circ ad_y(x)-\frac12(-ad_{[x,y]}^*(z)-ad_z^*([x,y]))$$
It is not clear that this can be simplified in any useful way.

As in the previous subsection, suppose that the inner product is of the form
$$\langle x,y\rangle=\langle\langle G^{-1} x,y\rangle\rangle$$
where $\langle\langle \cdot,\cdot\rangle\rangle$ is an $Ad$ invariant inner product.

Assuming this, $R(x,y)z$ equals
$$\frac 14(ad_{ad_y^*(z)}(x)+ad_{ad_z^*(y)}(x)-G\circ ad_{G^{-1}[y,z]}(x)+G\circ ad_{G^{-1} ad_y^*(z)}(x)+
G\circ ad_{G^{-1} ad_z^*(y)}(x)$$
$$-ad_{([y,z]-ad_y^*(z)-ad_z^*(y))}^*(x))$$
$$-\frac 14(-ad_y\circ G\circ ad_{G^{-1}z}(x)-ad_y\circ ad_z^*(x)+ad_y^*\circ ad_z(x)+ad_y^* \circ G\circ ad_{G^{-1}z}(x)+ad_y^*\circ ad_z^*(x)$$
$$+G \circ ad_{G^{-1}y}\circ ad_z(x)+G \circ ad_{G^{-1}y}\circ G\circ ad_{G^{-1}z}(x)+G \circ ad_{G^{-1}y}\circ ad_z^*(x))$$
$$-\frac 14 ad_z\circ ad_y(x)-\frac12(G\circ ad_{G^{-1}z}\circ ad_y(x)+ad_z^*\circ ad_y(x))$$

We expand further
$$=\frac 14(ad_{ad_y^*(z)}(x)+ad_{ad_z^*(y)}(x)-G\circ ad_{G^{-1}[y,z]}(x)+G\circ ad_{G^{-1} ad_y^*(z)}(x)+
G\circ ad_{G^{-1} ad_z^*(y)}(x)$$
$$-ad_{([y,z]-ad_y^*(z)-ad_z^*(y))}^*(x))$$
$$-\frac 14(-ad_y\circ G\circ ad_{G^{-1}z}(x)+ad_y\circ G\circ ad_z\circ G^{-1}(x)-G\circ ad_y\circ G^{-1}\circ ad_z(x)$$
$$-G\circ ad_y\circ ad_{G^{-1}z}(x)+G\circ ad_y\circ ad_z\circ G^{-1}(x))$$
$$+G \circ ad_{G^{-1}y}\circ ad_z(x)+G \circ ad_{G^{-1}y}\circ G\circ ad_{G^{-1}z}(x)+G \circ ad_{G^{-1}y}\circ ad_z^*(x)))$$
$$-\frac14 ad_z\circ ad_y(x)-\frac 12(G\circ ad_{G^{-1}z}\circ ad_y(x)+ad_z^*\circ ad_y(x))$$
Now reorder some of the terms before we try to find commutators:
$$=\frac 14(ad_{ad_y^*(z)}(x)+ad_{ad_z^*(y)}(x)-G\circ ad_{G^{-1}[y,z]}(x)+G\circ ad_{G^{-1} ad_y^*(z)}(x)+
G\circ ad_{G^{-1} ad_z^*(y)}(x)$$
$$-ad_{([y,z]-ad_y^*(z)-ad_z^*(y))}^*(x))$$
$$-\frac 14(-ad_y\circ G\circ ad_{G^{-1}z}(x)-G\circ ad_y\circ ad_{G^{-1}z}(x)+2G\circ ad_{G^{-1}z}\circ ad_y(x)$$
$$+ad_y\circ G\circ ad_z\circ G^{-1}(x)-G\circ ad_y\circ G^{-1}\circ ad_z(x)$$
$$+ad_z\circ ad_y(x)+G\circ ad_y\circ ad_z\circ G^{-1}(x)-2G\circ ad_z\circ G^{-1} \circ ad_y(x)$$
$$+G \circ ad_{G^{-1}y}\circ ad_z(x)+G \circ ad_{G^{-1}y}\circ G\circ ad_{G^{-1}z}(x)-G \circ ad_{G^{-1}y}\circ G\circ ad_z\circ G^{-1}(x))$$
(Now we look for commutators)
$$=\frac 14(ad_{ad_y^*(z)}(x)+ad_{ad_z^*(y)}(x)-G\circ ad_{G^{-1}[y,z]}(x)+G\circ ad_{G^{-1} ad_y^*(z)}(x)+
G\circ ad_{G^{-1} ad_z^*(y)}(x)$$
$$-ad_{([y,z]-ad_y^*(z)-ad_z^*(y))}^*(x))$$
$$-\frac 14(-[ad_y, G\circ ad_{G^{-1}z}](x)-G\circ [ad_y, ad_{G^{-1}z}](x)$$
$$+[ad_y, G]\circ ad_z\circ G^{-1}(x)+G\circ ad_y\circ [ad_z,G^{-1}](x)$$
$$+G\circ [G^{-1},ad_z]\circ ad_y(x)+G\circ [ad_y,ad_z]\circ G^{-1}+ G\circ ad_z \circ [ad_y, G^{-1}]$$
$$+G \circ ad_{G^{-1}y}\circ [ad_z\circ G^{-1},G](x)+G \circ ad_{G^{-1}y}\circ G\circ ad_{G^{-1}z}(x))$$

The key lines are those involving just $ad_y$ and $ad_z$. We simplify things one step at a time:
$$=\frac 14(ad_{ad_y^*(z)}(x)+ad_{ad_z^*(y)}(x)-G\circ ad_{G^{-1}[y,z]}(x)+G\circ ad_{G^{-1} ad_y^*(z)}(x)+
G\circ ad_{G^{-1} ad_z^*(y)}(x)$$
$$-ad_{([y,z]-ad_y^*(z)-ad_z^*(y))}^*(x))$$
$$-\frac 14(-[ad_y, G\circ ad_{G^{-1}z}](x)-G\circ [ad_y, ad_{G^{-1}z}](x)$$
$$+[ad_y, G]\circ ad_z\circ G^{-1}(x)+G\circ [ad_y, [ad_z,G^{-1}]](x)$$
$$G\circ [ad_y,ad_z]\circ G^{-1}+ G\circ ad_z \circ [ad_y, G^{-1}]$$
$$+G \circ ad_{G^{-1}y}\circ [ad_z\circ G^{-1},G](x)+G \circ ad_{G^{-1}y}\circ G\circ ad_{G^{-1}z}(x))$$

$$=\frac 14(ad_{ad_y^*(z)}(x)+ad_{ad_z^*(y)}(x)-G\circ ad_{G^{-1}[y,z]}(x)+G\circ ad_{G^{-1} ad_y^*(z)}(x)+
G\circ ad_{G^{-1} ad_z^*(y)}(x)$$
$$-ad_{([y,z]-ad_y^*(z)-ad_z^*(y))}^*(x))$$
$$-\frac 14(-[ad_y, G\circ ad_{G^{-1}z}](x)-G\circ [ad_y, ad_{G^{-1}z}](x)$$
$$+[ad_y, G\circ ad_z\circ G^{-1}]+G\circ [ad_y, [ad_z,G^{-1}]](x)$$
$$+G \circ ad_{G^{-1}y}\circ [ad_z\circ G^{-1},G](x)+G \circ ad_{G^{-1}y}\circ G\circ ad_{G^{-1}z}(x))$$

\begin{lemma}\label{lemma1} Suppose that the inner product is of the form
$$\langle x,y\rangle=\langle\langle G^{-1} x,y\rangle\rangle$$
where $\langle\langle \cdot,\cdot\rangle\rangle$ is an $Ad$ invariant inner product. Then as an operator on $x$, $R(x,y)z$ equals
$$=\frac 14(ad_{ad_y^*(z)}(x)+ad_{ad_z^*(y)}(x)-G\circ ad_{G^{-1}[y,z]}(x)+G\circ ad_{G^{-1} ad_y^*(z)}(x)+
G\circ ad_{G^{-1} ad_z^*(y)}(x)$$
$$-ad_{([y,z]-ad_y^*(z)-ad_z^*(y))}^*(x))$$
$$-\frac 14(-[ad_y, G\circ ad_{G^{-1}z}](x)-G\circ [ad_y, ad_{G^{-1}z}](x)$$
$$+[ad_y, G\circ ad_z\circ G^{-1}]+G\circ [ad_y, [ad_z,G^{-1}]](x)$$
$$+G \circ ad_{G^{-1}y}\circ [ad_z\circ G^{-1},G](x)+G \circ ad_{G^{-1}y}\circ G\circ ad_{G^{-1}z}(x))$$
\end{lemma}

\subsection{Mapping Groups}\

Consider the group $W^s(\Sigma;K)$. In the first part of this subsection we suppose that $G^{-1}=P^s$
where $P$ is a positive Laplace type operator.

\subsubsection{Order of Operators}\

To think about the order of $x\to R(x,y)z$ as a psd operator, we will write the operator in Lemma \ref{lemma1} in another way, and
we will also work on the first two lines: $R(x,y)z$ equals

$$\frac 14(ad_{ad_y^*(z)}(x)-G\circ ad_{ad_y^*(z)}\circ G^{-1}(x)+ad_{ad_z^*(y)}(x)-G\circ ad_{ad_z^*(y)}\circ G^{-1}(x)$$
$$-G\circ ad_{G^{-1}[y,z]}(x)+G\circ ad_{G^{-1} ad_y^*(z)}(x)+G\circ ad_{G^{-1} ad_z^*(y)}(x)$$
$$+G\circ ad_{[y,z]}\circ G^{-1}(x))$$
$$-\frac 14(-[ad_y, G\circ ad_{G^{-1}z}](x)-G\circ [ad_y, ad_{G^{-1}z}](x)$$
$$+[ad_y, G\circ ad_z\circ G^{-1}]+G\circ [ad_y, [ad_z,G^{-1}]](x)$$
$$+G \circ ad_{G^{-1}y}\circ [ad_z\circ G^{-1},G](x)+G \circ ad_{G^{-1}y}\circ G\circ ad_{G^{-1}z}(x))$$

$$=\frac 14(-G\circ [ad_{ad_y^*(z)}, G^{-1}](x)-G\circ [ad_{ad_z^*(y)},G^{-1}](x)$$
$$-G\circ ad_{G^{-1}[y,z]}(x)+G\circ ad_{G^{-1} ad_y^*(z)}(x)+G\circ ad_{G^{-1} ad_z^*(y)}(x)$$
$$+G\circ ad_{[y,z]}\circ G^{-1}(x))$$
$$-\frac 14(-[ad_y, G\circ ad_{G^{-1}z}](x)-G\circ [ad_y, ad_{G^{-1}z}](x)$$
$$+[ad_y, G\circ ad_z\circ G^{-1}]+G\circ [ad_y, [ad_z,G^{-1}]](x)$$
$$+G \circ ad_{G^{-1}y}\circ [ad_z\circ G^{-1},G](x)+G \circ ad_{G^{-1}y}\circ G\circ ad_{G^{-1}z}(x))$$

The two terms that individually are of order zero are
\begin{equation}\label{problem}
\frac14 (G\circ ad_{[y,z]}\circ G^{-1}-[ad_y, G\circ ad_z\circ G^{-1}])
\end{equation}
But this can be rewritten as
$$\frac14 ([G,ad_y]\circ ad_z\circ G^{-1}-G\circ ad_z\circ [ad_y,G^{-1}])$$

Thus as an operator on $x$, $R(x,y)z$ equals
\begin{equation}\label{curvop5}=\frac 14(-G\circ [ad_{ad_y^*(z)}, G^{-1}](x)-G\circ [ad_{ad_z^*(y)},G^{-1}](x)\end{equation}
$$-G\circ ad_{G^{-1}[y,z]}(x)+G\circ ad_{G^{-1} ad_y^*(z)}(x)+G\circ ad_{G^{-1} ad_z^*(y)}(x)$$
$$+[G,ad_y]\circ ad_z\circ G^{-1}-G\circ ad_z\circ [ad_y,G^{-1}])$$
$$-\frac 14(-[ad_y, G\circ ad_{G^{-1}z}](x)-G\circ [ad_y, ad_{G^{-1}z}](x)$$
$$  +G\circ [ad_y, [ad_z,G^{-1}]](x)$$
$$+G \circ ad_{G^{-1}y}\circ [ad_z\circ G^{-1},G](x)+G \circ ad_{G^{-1}y}\circ G\circ ad_{G^{-1}z}(x))$$

Now all of the individual terms have order $\le -1$. In the case of the circle it has been checked
that the order is $-1$ (see \cite{Driver}).

\subsubsection{The trace over $\mathfrak k$}\

Following the strategy of Freed, we now consider first taking
the trace over $\mathfrak k$ of the operator $x\to R(x,y)z$.

If we identify $W^s(X,\mathfrak k)$ with the tensor product $W^s(X)\otimes \mathfrak k$, it is clear
that $G$ and the trace of $\mathfrak k$ commute. Also for $\chi\in\mathfrak k$, $ad_{\chi}$ is nilpotent, and hence is traceless.
Consequently the trace over $\mathfrak k$ kills the terms in the first two lines of (\ref{curvop5}) and the second term
in the fourth line. Now, the first summand in \ref{problem} has trace zero.  Consequently $trace_{\mathfrak k}(x\to R(x,y)z)$ equals $trace_{\mathfrak k}$ of the following operator of $x$,
\begin{equation}\label{oper1}-\frac 14(-[ad_y, G\circ ad_{G^{-1}(z)}](x)+G\circ [ad_y, [ad_z,G^{-1}]](x)+[ad_y, G\circ ad_z\circ G^{-1}](x)\end{equation}
$$+G \circ ad_{G^{-1}(y)}\circ [ad_z\circ G^{-1},G](x)+G \circ ad_{G^{-1}(y)}\circ G\circ ad_{G^{-1}(z)}(x))$$

The following is a basic observation of Freed:

\begin{lemma}This (scalar) operator has psd order $\le \max(-2,-2s)$ \end{lemma}

\begin{proof}Suppose that $y=Y\otimes a$ and $z=Z\otimes b$. For brevity identify $Y$ with the corresponding multiplication operator
and so on. In the following expression for the operator (\ref{oper1}) acting on $x$, $G^{-1}(z)$ is viewed
as a multiplication operator (consequently we cannot cancel $G$ and $G^{-1}$ in the first line, and so on - this is a weakness of
our notation):
$$-\frac 14(-Y\circ G\circ G^{-1}(Z)\otimes ad_a\circ ad_b+G\circ G^{-1}(Z)\circ Y\otimes ad_b\circ ad_a
+Y\circ G\circ Z\circ G^{-1}\otimes ad_a\circ ad_b-G\circ Z\circ G^{-1}\circ Y\otimes ad_b\circ ad_a$$
$$
+G\circ Y\circ Z \circ G^{-1}\otimes ad_a\circ ad_b-G\circ Y\circ G^{-1}\circ Z\otimes ad_a\circ ad_b -G\circ Z\circ G^{-1} \circ Y\otimes ad_b\circ ad_a+ G\circ G^{-1}\circ Z\circ Y\otimes ad_b\circ ad_a     $$
$$+G \circ G^{-1}(Y)\circ Z\circ G^{-1}\circ G\otimes ad_a\circ ad_b-G \circ G^{-1}(Y)\circ G\circ Z\circ G^{-1}\otimes ad_a\circ ad_b$$
$$G \circ G^{-1}(Y)\circ G\circ G^{-1}(Z)\otimes ad_a\circ ad_b)$$
When we take the trace over $\mathfrak k$, this equals $-\frac14 \kappa(a,b)$ times the operator
$$-Y\circ G\circ G^{-1}(Z)+G\circ G^{-1}(Z)\circ Y
+Y\circ G\circ Z\circ G^{-1}-G\circ Z\circ G^{-1}\circ Y$$
$$
+G\circ Y\circ Z \circ G^{-1}-G\circ Y\circ G^{-1}\circ Z -G\circ Z\circ G^{-1}\circ Y+  Z\circ Y $$
$$+G \circ G^{-1}(Y)\circ Z-G \circ G^{-1}(Y)\circ G\circ Z\circ G^{-1}$$
$$+G \circ G^{-1}(Y)\circ G\circ G^{-1}(Z)$$
(Note the first two terms can be written as a commutator (and hence this commutator will have order $\le -2s-1$),
and the eighth and ninth terms can similarly be expressed as a commutator):
$$=[G\circ G^{-1}(Z), Y]
+Y\circ G\circ Z\circ G^{-1}-G\circ Z\circ G^{-1}\circ Y$$
$$
+G\circ Y\circ Z \circ G^{-1}-G\circ Y\circ G^{-1}\circ Z -G\circ Z\circ G^{-1}\circ Y+  Z\circ Y $$
$$+G \circ G^{-1}(Y)\circ [G^{-1},G\circ Z]$$
$$+G \circ G^{-1}(Y)\circ G\circ G^{-1}(Z)$$

To understand why there is reduction of order, consider the terms that just involve $Y,Z$,
$$Y\circ G\circ Z\circ G^{-1}-G\circ Z\circ G^{-1}\circ Y
+G\circ Y\circ Z \circ G^{-1}-G\circ Y\circ G^{-1}\circ Z -G\circ Z\circ G^{-1}\circ Y+  Z\circ Y $$
This equals
$$(G\circ Y+[Y,G])\circ Z\circ G^{-1}-G\circ Z\circ (Y\circ G^{-1}+[G^{-1},Y])$$
$$+G\circ Y\circ Z\circ G^{-1}-G\circ Y\circ (Z\circ G^{-1}+[G^{-1},Z])-G\circ (G^{-1}\circ Z+[Z,G^{-1}])\circ Y+Z\circ Y$$
$$=[Y,G]\circ Z\circ G^{-1}-G\circ Z\circ [G^{-1},Y])-G\circ Y\circ [G^{-1},Z]-G\circ [Z,G^{-1}])\circ Y$$
$$=G\circ [G^{-1},Y]\circ G\circ (G^{-1}\circ Z+[Z,G^{-1}])-G\circ Z\circ [G^{-1},Y])-G\circ Y\circ [G^{-1},Z]-G\circ [Z,G^{-1}]\circ Y$$

$$=G\circ [G^{-1},Y]\circ G\circ [Z,G^{-1}]+G\circ [[G^{-1},Y],Z]-G\circ [Y,[G^{-1},Z]]$$
$$=G\circ [G^{-1},Y]\circ G\circ [Z,G^{-1}]+2G\circ [[G^{-1},Y],Z]$$
This has order $\le -2$.

All of the other terms have order $\le -2s$. This proves Freed's observation.
\end{proof}

\section{Dixmier/Wodzicki Regularized Ricci Curvature when $\Sigma$ is a Surface}

We now suppose that $\Sigma$ is a surface with Riemannian metric. We continue to assume that
$P$ is a positive Laplace type operator and  $G^{-1}=P^s$.

We have to calculate ($-\frac14$ times) the Wodzicki residue of the operator
$$[G\circ G^{-1}(Z),Y]+G\circ [G^{-1},Y]\circ G\circ [Z,G^{-1}]+2G\circ [[G^{-1},Y],Z]
 $$
$$+G \circ G^{-1}(Y)\circ G\circ [G^{-1},Z]+G \circ G^{-1}(Y)\circ G\circ G^{-1}(Z)$$

When $s\ge 1$, all except the second and third terms have order $<-2$. Hence we must calculate ($-\frac14$ times) the Wodzicki residue of
\begin{equation}\label{eqn10}G\circ [G^{-1},Y]\circ G\circ [Z,G^{-1}]+2G\circ [[G^{-1},Y],Z]
 \end{equation}

For the first term in (\ref{eqn10}), we need to find the principal symbol of $[G^{-1},Y]$. This equals the Poisson bracket of the two functions on $T^*\Sigma$, $(g^{-1})^s:p\to p^{2s}$ and $Y$. Abstractly this equals
$$s(g^{-1})^{s-1}\pi(d(g^{-1})\wedge dY)$$ where $\pi$ denotes the Poisson tensor.
In coordinates $(x,p)$ (using summation convention)
$$d(g^{-1})=(dg^{ij})p_ip_j+2g^{ij}p_idp_j \text{ and } dY=\frac{\partial Y}{\partial x^k}dx^k.$$
Using $\{p_i,x_j\}=\delta_{ij}$, it follows that the principal symbol of $[G^{-1},Y]$ equals
$$-2s(g^{-1})^{s-1}g^{-1}(p,dY)$$
Consequently the principal symbol of the first term in (\ref{eqn10}) equals
\begin{equation}\label{eqn11}-4s^2g^{-1}(p,dY)g^{-1}(p,dZ)\end{equation}

For the second term in (\ref{eqn10}) we need to compute (twice the) Poisson bracket
\begin{equation}\label{eqn2}{2s(g^{-1})^{s-1}\pi(g^{-1}(p,dY),Z})=\pi(d(2s(g^{-1})^{s-1}g^{-1}(p,dY)),dZ)\end{equation}
Now
$$d((g^{-1})^{s-1}g^{-1}(p,dY))=(s-1)(g^{-1})^{s-2}\left((dg^{ij})p_ip_j+2g^{ij}p_idp_j\right)g^{-1}(p,dY)$$
$$
+(g^{-1})^{s-1}(dg^{ij})p_i\frac{\partial Y}{\partial x^j}+g^{ij}dp_i\frac{\partial Y}{\partial x^j}+ g^{ij}p_id(\frac{\partial Y}{\partial x^j}))$$
When we compute the Poisson bracket, we only need the $dp_i$ terms. It follows that (\ref{eqn2}) equals
\begin{equation}\label{eqn12}4s(s-1)(g^{-1})^{s-2}g^{-1}(p,dY)g^{-1}(p,dZ)+2sg^{-1}(dY,dZ)\end{equation}

These two calculations imply that the leading symbol of (\ref{10}) equals
\begin{equation}\label{leadingsymbol}-4s^2|p|^{2s-4}g^{-1}(p,dY)g^{-1}(p,dZ)+8s(s-1)(g^{-1})^{s-2}g^{-1}(p,dY)g^{-1}(p,dZ)+
4s|p^{2s}|g^{-1}(dY,dZ)\end{equation}

To calculate the Wodzicki residue, we have to multiply this homogeneous function on $T^*\Sigma$ by the Riemannian volume form for
the sphere bundle $S(T^*\Sigma)$, then integrate. Consider the first symbol in (\ref{leadingsymbol}. To integrate this over the sphere bundle (against the volume form), we first
integrate over the unit circle in the cotangent space at a point on the surface, then we integrate over the surface. Fix a point on the surface, say $q_0$. Let $S$ denote the unit circle in the cotangent space at $q_0$, and let $\theta$ denote an arclength parameter (relative to the metric defined by $g^{-1}$). Because $g^{-1}$ (viewed as a function on the cotangent bundle) equals unity on this circle, we obtain $-4s^2$ times
$$\int_S \langle dY,\left(\begin{matrix}cos(\theta)\\sin(\theta)\end{matrix}\right)\rangle\langle dZ,\left(\begin{matrix}cos(\theta)\\sin(\theta)\end{matrix}\right)\rangle d\theta$$
$$=\int_S \left(((dY)_1 cos(\theta)+(dY)_2 sin(\theta))((dZ)_1 cos(\theta)+(dZ)_2 sin(\theta))\right)d\theta$$
(where $(dY)_1$ denotes the first component of $dY$ in the orthonormal system for the metric, and so on)
$$=\int_S \left(((dY)_1(dZ)_1 cos(\theta)^2+(dY)_1 (dZ)_2 cos(\theta)sin(\theta))+(dY)_2(dZ)_1 sin(\theta)cos(\theta)+(dY)_2(dZ)_2 sin(\theta)^2)\right)d\theta$$
$$=\pi \langle dY,dZ\rangle\vert_{q_0}$$
where we have used $\int_S cos(\theta)sin(\theta)d\theta=0$ and $\int_S cos(\theta)^2d\theta=\int_S sin(\theta)^2d\theta=\pi$.
When we integrate this over the surface, using the area form, we obtain
$$\pi\int_{\Sigma}dY\wedge *dZ$$ The second term is similar.
For the third term in (\ref{leadingsymbol}), the integrand does not depend on the circle coordinate,
so we obtain $4s$ times
$$2\pi\int_{\Sigma}dY\wedge * dZ$$

It follows that
$$Ric(Y\otimes b,Z\otimes c)=-\kappa(b,c)\pi s^2 \int_{\Sigma} dY\wedge * dZ$$
We summarize this as follows.

\begin{proposition} Suppose that $\Sigma$ is a closed Riemannian surface and $s>1$. Using the two step procedure of first calculating the trace over $\mathfrak k$ and then using the Dixmier trace to define what we mean by the Ricci curvature for $W^s(\Sigma,K)$,
$$Ric(y,z)=-\pi s^2 \int_{\Sigma} \kappa(dy\wedge * dz)$$
\end{proposition}

Note that because $\kappa$ is negative, Ricci curvature is positive, and becomes more positive as $s$ increases.
Note also that the result is independent of the specific form of $P$, beyond the fact that it is a Laplace type operator.

\subsection{The Critical Case}\

We continue to assume that $\Sigma$ is a closed Riemannian surface, and pick a basepoint $z_0\in\Sigma$.
If we restrict to reasonably nice maps, there is a bijection
$$\Omega(\Sigma;\mathfrak k) \leftrightarrow Map(\Sigma,\mathfrak k)/\mathfrak k $$
$$x \leftrightarrow x \text{ modulo } \mathfrak k $$
The first space is a Lie algebra, and (momentarily assuming that $\mathfrak k$ is simple) there is an essentially unique $PSL(2,\mathbb C) \times Ad(K)$ invariant inner product on the latter space,
$$\langle x,x\rangle= -\int_{\Sigma} \kappa ( dx\wedge *dx)$$
The essential uniqueness of the inner product on $Map(\Sigma,\mathfrak k)/\mathfrak k$ is a consequence of Schur's lemma and the fact that $PSL(2,\mathbb C)\times Ad(K)$ acts unitarily and irreducibly on an appropriate subspace (this action of $PSL(2,\mathbb C)$ is a member of the principal series for this group).
Now it is not the case that $W^1(\Sigma,K)$ is a Lie group, and consequently it is not possible to smoothly translate the (when $\mathfrak k$ is simple, essentially unique) $PSL(2,\mathbb C)\times Ad(K)$ invariant inner product around the quotient $W^1(\Sigma,K)/K$. However,
as in Freed, we can consider the restriction to smoother loops ($W^s(\Sigma,K)/K$, for any $s>1$; in this context the metric is not complete), and we can use the two step procedure of first calculating the trace over $\mathfrak k$ and then using the Dixmier trace, to define what we mean by the Ricci curvature. With this understood, we obtain the following

\begin{corollary} $W^1(\Sigma,K)/K$ is an Einstein manifold with positive constant, $\pi$.\end{corollary}

\end{document}